\documentclass[11pt]{article}
\usepackage{amsmath}
\usepackage{amstext,amsbsy}
\usepackage{epsfig,multicol}
\usepackage{amsthm}
\usepackage{amssymb,latexsym}
\usepackage{color}
\usepackage[pagebackref,hypertexnames=false, colorlinks, citecolor=blue, linkcolor=red]{hyperref}

\theoremstyle{plain}
\newtheorem{theorem}{Theorem}[section]
\newtheorem{lemma}[theorem]{Lemma}
\newtheorem{corollary}[theorem]{Corollary}
\newtheorem{proposition}[theorem]{Proposition}
\newtheorem{example}[theorem]{Example}

\theoremstyle{definition}

\newtheorem*{properties*}{Properties}

\newenvironment{definition*}[1][Definition]{\begin{trivlist}
\item[\hskip \labelsep {\bfseries #1}]}{\end{trivlist}}

\numberwithin{equation}{section}

\newcommand{\inv}{\mathrm{inv}}
\newcommand{\maj}{\mathrm{maj}}

\newcommand{\nn}{\mathrm{ne}}
\newcommand{\cc}{\mathrm{cr}}
\newcommand{\al}{\mathrm{al}}

\newcommand{\Rlminl}{\mathrm{Rlminl}}
\newcommand{\Lrmaxp}{\mathrm{Lrmaxp}}

\newcommand{\Short}{\mathrm{Short}}
\newcommand{\Long}{\mathrm{Long}}
\newcommand{\Cyc}{\mathrm{Cyc}}
\newcommand{\cyc}{\mathrm{cyc}}
\newcommand{\sor}{\mathrm{sor}}
\newcommand{\rlmin}{\mathrm{rlmin}}
\newcommand{\col}{\mathrm{col}}

\begin{document}

\title{Cycles and sorting index for matchings and restricted permutations}
\author{
Svetlana~Poznanovi\'{c}$^{1}$
 \vspace{.2cm} \\
School of Mathematics\\
Georgia Institute of Technology } 
\date{}
\maketitle
\begin{abstract} 
We prove that the Mahonian-Stirling pairs of permutation statistics $(\sor, \cyc)$ and $(\inv, \mathrm{rlmin})$ are equidistributed on the set of permutations that correspond to arrangements of $n$ non-atacking rooks on a Ferrers board with $n$ rows and $n$ columns. The proofs are combinatorial and use bijections between matchings and Dyck paths and a new statistic, sorting index for matchings, that we define. We also prove a refinement of this equidistribution result which describes the minimal elements in the permutation cycles  and the right-to-left minimum letters. Moreover, we define a sorting index for bicolored matchings and use it to show analogous equidistribution results for restricted permutations of type $B_n$ and $D_n$.
\end{abstract}



{\renewcommand{\thefootnote}{}
\footnote{\emph{E-mail address}: 
svetlana@math.gatech.edu.} } 
\footnotetext[1]{The author was supported in part by a Burroughs
Wellcome Fund Career Award at the Scientific Interface to C.E.~Heitsch.}

\section{Introduction} \label{S:introduction}

An inversion in a permutation $\sigma$ is a pair $\sigma(i) > \sigma(j)$ such that $i <j$. The number of inversions in $\sigma$ is denoted by $\inv(\sigma)$. The distribution of $\inv$ over the symmetric group $S_n$ was first found by Rodriguez~\cite{Rodriguez} in 1837 and is well known to be \[ \sum_{\sigma \in S_n}  q^{\mathrm{inv}(\sigma)} = (1+q)(1+q+q^2) \cdots (1+ q + \cdots + q^{n-1}).\]   Much later, MacMahon~\cite{Mac} defined the major index $\mathrm{maj}$ and proved that it has the same distribution as $\mathrm{inv}$. In his honor, all permutation statistics that are equally distributed with $\inv$ are called Mahonian. MacMahon's remarkable result initiated a systematic research of permutation statistics and in particular many more Mahonian statistics have been described in the literature since then.

Another classical permutation statistic is the number of cycles, $\cyc$. Its distribution is given by
\[ \sum_{\sigma \in S_n}  t^{\cyc(\sigma)} = t(t+1)(t+2) \cdots (t+n-1)\] and the coefficients of this polynomial are known as the unsigned Stirling numbers of the first kind.

Given these two distributions, it is natural then to ask which ``Mahonian-Stirling'' pairs of statistics $(\mathrm{stat}_1, \mathrm{stat}_2)$ have the distribution 
\begin{equation} \label{Sn} \sum_{\sigma \in S_n}  q^{\mathrm{stat}_1(\sigma)}  t^{\mathrm{stat}_2(\sigma)} = t(t+q)(t+q+q^2) \cdots (t+ q + \cdots + q^{n-1}).\end{equation}
As proved by Bj\"{o}rner and Wachs~\cite{BW}, $(\inv, \rlmin)$ and $(\maj, \rlmin)$ are two such pairs, where $\rlmin$ is the number of right-to-left minimum letters. A right-to-left minimum letter of  a permutation $\sigma$ is a letter $\sigma(i)$ such that $\sigma(i)< \sigma(j)$ for all $j > i$. The set of all right-to-left minimum letters in $\sigma$ will be denoted by $\Rlminl(\sigma)$. In fact, Bj\"{o}rner and Wachs proved the following stronger result
\begin{equation} \label{BW} \sum_{\sigma \in S_n} q^{\inv(\sigma)} \prod_{i \in \Rlminl(\sigma)} t_i =  \sum_{\sigma \in S_n} q^{\maj(\sigma)} \prod_{i \in \Rlminl(\sigma)} t_i = t_1(t_2+q)(t_3+q+q^2)\cdots(t_n +q + \cdots + q^{n-1}).\end{equation} 



A natural Mahonian partner for $\cyc$ was found by Petersen~\cite{Petersen}. For a given permutation $\sigma \in S_n$ there is a unique expression $$ \sigma = (i_1 j_1)(i_2 j_2)\cdots (i_k j_k)$$ as a product of transpositions such that $i_s < j_s$ for $1 \leq s \leq k$ and $j_1 < \cdots < j_k$. The sorting index of $\sigma$ is defined to be
\[ \sor(\sigma) = \sum_{s=1}^k (j_s - i_s).\] The sorting index can also be described as the total distance the elements in $\sigma$ travel when $\sigma$ is sorted using the Straight Selection Sort algorithm~\cite{Knuth} in which, using a transposition, we move the largest number to its proper place, then the second largest to its proper place, etc. For example, the steps for sorting $\sigma = 6571342$ are
\[ 65{\bf{7}}1342 \xrightarrow{(37)} {\bf{6}}521347 \xrightarrow{(1 6)}  4{\bf{5}}21367 \xrightarrow{(2 5)} {\bf{4}}321567 \xrightarrow{(1 4)} 1{\bf{3}}24567 \xrightarrow{(2 3)} 1234567 \]
and therefore $\sigma = (2 \; 3)(1 \; 4)(2 \; 5)(1 \; 6)(3 \; 7)$ and $\sor(\sigma) = (3-2) + (4-1) + (5-2) + (6-1) + (7-3) = 16$.  The relationship to other Mahonian statistics and the Eulerian partner for $\sor$ were studied  by Wilson~\cite{Wilson} who called the sorting index $\mathrm{DIS}$.

Petersen showed that
\[ \sum_{\sigma \in S_n}  q^{\sor(\sigma)} t^{\cyc(\sigma)}= t(t+q)(t+q+q^2) \cdots (t+ q + \cdots + q^{n-1}), \]which implies equidistribution of the pairs $(\inv, \mathrm{rlmin})$ and $(\sor,  \cyc)$. 

In this article we show that the pairs  $(\inv, \rlmin)$ and $(\sor, \cyc)$ have the same distribution on the set of restricted permutations 
\[ S_{\bf{r}} =\{\sigma \in S_n : \sigma(k) \leq r_k, 1 \leq k \leq n  \}  \] for a nondecreasing sequence of  integers $1 \leq r_1 \leq r_2 \leq \cdots \leq r_n \leq n$. These can be described as permutations that correspond to arrangements of $n$ non-atacking rooks on a Ferrers board with rows of length $r_1, \dots, r_n$.  To obtain the results, in Section~\ref{mat} we define a sorting index and cycles for perfect matchings and study the distributions of these statistics over matchings of fixed type. We use bijections between matchings and weighted Dyck paths which enable us to keep track of set-valued statistics and obtain more refined results similar to~\eqref{BW} for restricted permutations.

Analogously to $\sor$, Petersen defined sorting index for signed permutations of type $B_n$ and $D_n$. Using algebraic methods he proved that
\begin{equation} \label{petb} \sum_{\sigma \in B_n} q^{\sor_B(\sigma)} t^{\ell'_B(\sigma)} = \sum_{\sigma \in B_n} q^{\inv_B(\sigma)} t^{\mathrm{nmin}_B(\sigma)} = \prod_{i=1}^n (1+t[2i]_q-t),\end{equation} where $\ell'_B(\sigma)$ is the reflection length of $\sigma$, i.e., the minimal number of transpositions 
in $$\{(i j) : 1 \leq i < j \leq n \} \cup \{(\bar{i} j) : 1 \leq i < j \leq n \}$$ needed to represent $\sigma$; $\inv_B(\sigma)$
 is the number of type $B_n$ inversions, which is known to be equal to the length of $\sigma$ and is given by 
 \begin{equation} \label{invb}  \inv_B (\sigma) = |\{1 \leq i < j \leq n: \sigma(i) > \sigma(j) \}| + |\{1 \leq i < j \leq n: -\sigma(i) > \sigma(j) \} | + N(\sigma), \end{equation} where
\[ N(\sigma) = \text{ number of negative signs in } \sigma. \] Finally,
\begin{equation} \label{nmin} \mathrm{nmin}_B(\sigma) = |\{i : \sigma(i) > |\sigma(j)| \text{ for some } j > i\}| + N(\sigma). \end{equation}
Petersen also defined $\sor_D$, a sorting index for type $D_n$ permutations and showed that it is equidistributed with the number of type $D_n$ inversions:
\begin{equation} \label{petd} \sum_{\sigma \in D_n} q^{\sor_D(\sigma)} = \sum_{\sigma \in D_n} q^{\inv_D(\sigma)}  = [n]_q \cdot \prod_{i=1}^{n-1} [2i]_q.\end{equation}

In Section~\ref{bimat} we define a sorting index and cycles for bicolored matchings and give a combinatorial proof that the pairs $(\sor_B, \ell'_B)$ and $(\inv_B, \mathrm{nmin}_B)$ are equidistributed on the set of restricted signed permutations 
\[ B_{\bf{r}} =\{\sigma \in B_n : |\sigma(k)| \leq r_k, 1 \leq k \leq n  \}  \] for a nondecreasing sequence of  integers $1 \leq r_1 \leq r_2 \leq \cdots \leq r_n \leq n$. Using bijections between bicolored matchings and weighted Dyck paths with bicolored rises, we in fact prove equidistribution of set-valued statistics and their generating functions. Moreover, we find natural Stirling partners for $\sor_D$ and $\inv_D$ and prove equidistribution of the two Mahonian-Stirling pairs on sets of restricted permutations of type $D_n$:
\[ D_{\bf{r}} =\{\sigma \in D_n : |\sigma(k)| \leq r_k, 1 \leq k \leq n  \}.  \]

\section{Statistics on perfect matchings}
\label{mat}

A matching is a partition of a set in blocks of size at most two and if it has no single-element blocks the matching is said to be perfect. The set of all perfect matchings with $n$ blocks is denoted by $\mathcal{M}_n$. All matchings in this work will be perfect and henceforth we will omit this adjective. 

\subsection{Statistics based on crossings and nestings}

A matching in $\mathcal{M}_n$ can be represented by a graph with $2n$ labeled vertices and $n$ edges in which each vertex has a degree $1$. The vertices $1, 2, \dots, 2n$ are drawn on a horizontal line in natural order and two vertices that are in a same block are  connected by a semicircular arc in the upper half-plane. We will use $i \cdot j$ to denote an arc with vertices $i < j$. The vertex $i$ is said to be the opener while  $j$ is said to be the closer of the arc. For a vertex $i$, we will denote by $M(i)$ the other vertex which is in the same block in the matching $M$ as $i$. Two arcs $i \cdot j$ and $k \cdot l$ with $i < k$ can be in three different relative positions. We say that they form a crossing if $i < k < j< l$, they form a nesting if $ i < k < l < j$, and they form an alignment if $i <j < k < l$. The arc with the smaller opener will be called the left arc of the crossing, nesting, or the alignment, respectively, while the arc with the larger opener will be called the right arc. The numbers of crossings, nestings, and alignements in a matching $M$ are denoted by $\cc(M)$, $\nn(M)$, and $\al(M)$, respectively. 

If $o_1 < \cdots < o_n$ and $c_1 < \cdots <c_n$ are the openers and the closers in $M$, respectively, let $$\Long(M)= \{ k: o_k  \cdot M(o_k) \text{ is not a right arc in a nesting}\}$$ and $$\Short(M)= \{ k: M(c_k) \cdot c_k \text{ is not a left arc in a nesting}\}.$$ Similarly, let $$\mathrm{Left}(M)= \{ k: o_k  \cdot  M(o_k) \text{ is not a right arc in a crossing}\}.$$ We will use lower-case letters to denote the cardinalities of the sets. For example, $\mathrm{long}(M)= | \Long(M)|$. 
\begin{example} For the matching $M$ in Figure~\ref{fbijone} we have $\nn(M) = \cc(M) = \al(M) =5$, $\mathrm{Long}(M) = \{1, 2\}$, $\mathrm{Short}(M) = \{1, 2, 3, 5\}$, and $\mathrm{Left}(M) = \{1, 5\}$.
\end{example}
The pair of sets $(\{o_1, \dots, o_n\}, \{c_1, \dots, c_n\})$ of openers and closers of a matching $M$ is called the type of $M$. There is a natural one-to-one correspondence between types of matchings in $\mathcal{M}_n$ and Dyck paths of semilength $n$, i.e.,  lattice paths that start at $(0,0)$, end at $(2n,0)$, use steps $(1,1)$ (rises) and $(1,-1)$ (falls), and never go below the $x$-axis. The set of all such Dyck paths will be denoted by $\mathcal{D}_n$. Namely, the openers in the type correspond to the rises in the Dyck path while the closers correspond to the falls. Therefore, for convenience, we will  say that a matching in $\mathcal{M}_n$ is of type $D$, for some Dyck path $D \in \mathcal{D}_n$, and we will denote the set of all  matchings of type $D$ by $\mathcal{M}_n(D)$.

The height of a rise of a Dyck path is the $y$-coordinate of the right endpoint of the corresponding $(1,1)$ segment. The sequence $(h_1, \dots, h_n)$ of the heights of the rises of a $D \in \mathcal{D}_n$ when read from left to right will be called shortly the height sequence of $D$. For example, the height sequence of the Dyck path in Figure~\ref{fbijone} is $(1, 2, 3, 3, 3, 4)$. A weighted Dyck path is a pair $(D, (w_1, \dots, w_n))$ where $D \in \mathcal{D}_n$ with height sequence $(h_1, \dots, h_n)$ and $w_i \in \mathbb{Z} \text{ with } 1 \leq w_i \leq h_i$.  There is a well-known bijection $\varphi_1$ from  the set $\mathcal{WD}_n$ of weighted Dyck paths of semilength $n$ to $\mathcal{M}_n$~\cite{deSC}.   Namely, the openers $o_1 < o_2 < \cdots < o_n$ of the matching  that corresponds to a given  $(D,(w_1, \dots, w_n)) \in \mathcal{WD}_n$ are determined according to the type $D$.  To construct the corresponding matching $M$, we connect the openers from right to left, starting from $o_n$. After $o_n, o_{n-1}, \dots, o_{k+1}$ are connected to a closer, there are exactly $h_k$ unconnected closers that are larger than $o_k$. We connect $o_k$ to the $w_k$-th of the available closers, when they are listed in decreasing order (see Figure~\ref{fbijone}).

\begin{figure}[h]
\begin{center}
\includegraphics[height=2cm]{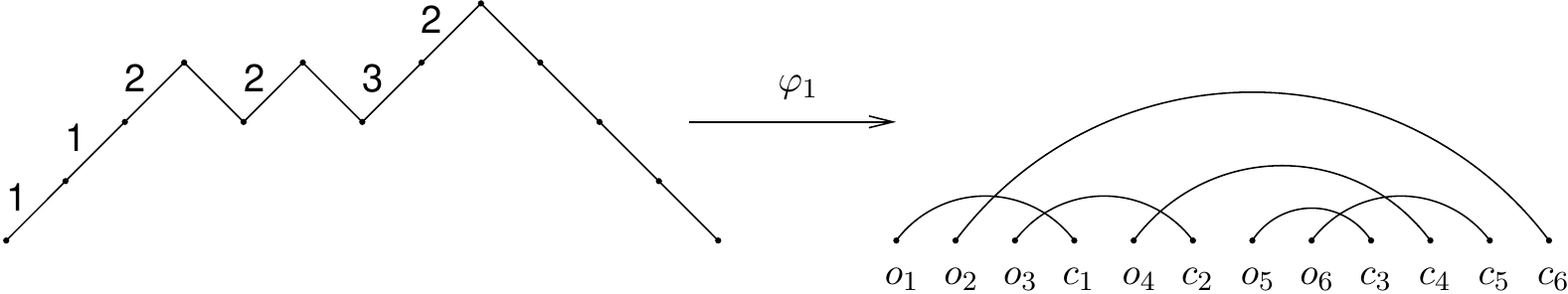}
\end{center}
\caption{The bijection $\varphi_1$ between weighted Dyck paths and matchings.}
\label{fbijone}
\end{figure}

Via the bijection $\varphi_1$ we immediately get the following generating function.
\begin{theorem} 
\label{thmbasic}
If $D \in \mathcal{D}_n$ has a height sequence  $(h_1, \dots, h_n)$, then
\begin{equation} \label{eqone}
\sum_{M \in \mathcal{M}_n(D)} p^{\cc(M)} q^{\nn(M)} \prod_{i \in \mathrm{Left}(M)} s_i \prod_{i \in \mathrm{Long}(M)} t_i = \prod_{k=1}^n (t_k  p^{h_k-1}   +  p^{h_k -2}q   + \cdots + p q^{h_k-2} + s_k q^{h_k-1}). 
\end{equation}
\end{theorem}

\begin{proof} The edge $o_k \cdot M(o_k)$ will be a right arc in  exactly $w_{k}-1$ nestings and exactly $h_k-w_k$ crossings in $M = \varphi_1(D, (w_1, \dots, w_n))$.  So, $k \in \mathrm{Long}(M)$ if and only if $w_k=1$ while the closer that is connected to $o_k$ is in $\mathrm{Left}(M)$ if and only if $w_k = h_k$. 
\end{proof}

The map $\varphi_1$ also has the following property. The definition of $\Rlminl$ was given for permutations but it extends to words in a straightforward way. 

\begin{proposition} \label{propone} Let $(D, (w_1,\dots, w_n)) \in \mathcal{WD}_n$  and $M= \varphi_1(D, (w_1,\dots, w_n))$. Then
\begin{equation}\label{prop1} \Short(M) = \Rlminl(2-w_1, 3-w_2, \dots, n+1-w_n). \end{equation}
\end{proposition}

\begin{proof} The proof is by induction on $n$, the number of arcs in the matching. If $n=1$, the only matching with one arc is $M= \{o_1 \cdot c_1\}$, and $\Short(M) = \{1\}$. The corresponding weighted Dyck path has only one rise with weight $w_1=1$. So, $\Rlminl(2-w_1) = \Rlminl(1)=\{1\}$. 

Suppose~\eqref{prop1} holds for all matchings with $n-1$ arcs. If $M$ is a matching with $n$ arcs, openers $o_1 < \cdots < o_n$ and closers $c_1 < \cdots < c_n$, let $M'$ be the matching obtained from $M$ by deleting the arc $o_n \cdot M(o_n)$. The weight sequence associated to $M'$ via the map $\varphi_1^{-1}$ is $(w_1, \dots, w_{n-1})$, since $w_k -1$ is the number of nestings in which the arc of the $k$-th opener is a right arc, and this number is the same in both $M$ and $M'$. Not also that  $M(o_n) = c_{n+1-w_n}$.

Let the closers in $M'$ be $c'_1 < \cdots < c'_{n-1}$. Then for $ i < n$ and $1 \leq k <  n+1-w_n$, $o_i \cdot c_k$ is an arc in $M$ if and only if $o_i \cdot c'_k$ is an arc in $M'$. On the other hand, for $  n+1-w_n < k \leq n$,  $o_i \cdot c_k$ is an arc in $M$ if and only if $o_i \cdot c'_{k-1}$ is an arc in $M'$.

For a number $k \in [n]$ there are two possibilities:
\begin{enumerate}
\item $k \in \Rlminl(2-w_1, 3-w_2, \dots, n+1-w_n)$

If $k = n+1-w_n$ then $k \in \Short(M)$ because $o_n \cdot c_{n+1-w_n}$ is an arc in $M$ and there are no arcs nested below it.

If $k \neq n+1-w_n$ then necessarily $k < n + 1-w_n$. Also $k \in \Rlminl(2-w_1, 3-w_2, \dots, n-w_{n-1}) = \Short(M')$, which implies that the arc $M'(c'_k) \cdot c'_k$ in  $M'$  has no arcs nested below it. But then $M(c_k) \cdot c_k$ is an arc in $M$ and  the  additional arc $o_n \cdot c_{n+1-w_n}$ in $M$ is not nested below it. So, $k \in \Short(M)$. 

\item $k \notin \Rlminl(2-w_1, 3-w_2, \dots, n+1-w_n)$

Necessarily, $k \neq n+1-w_n$.

If $k < n+1-w_n$, then $k \notin \Rlminl(2-w_1, 3-w_2, \dots, n-w_{n-1})$ and, by the induction hypothesis, there is an arc $o_r \cdot c'_s$ nested below $M'(c'_k) \cdot c'_k$ in $M'$. But then the arc $o_r \cdot c_s$ is  nested below $M(c_k) \cdot c_k$ in $M$, and consequently, $k \notin \Short(M)$. 

If  $k > n+1-w_n$, then since $o_n$ is the largest opener in $M$ and $c_k > c_{n+1-w_n}$, the arc $o_n \cdot c_{n+1-w_n}$ is nested below $M(c_k) \cdot c_k$, and so $k \notin \Short(M)$.
\end{enumerate}
\end{proof}


\subsection{Cycles and sorting index for matchings}

Let $M_0$ be a matching in $\mathcal{M}_n(D)$.  For $M \in \mathcal{M}_n(D)$ define $\cyc(M, M_0)$  as the number of cycles in the graph $G=(M, M_0)$ on $2n$ vertices in which the arcs from $M$ are drawn in the upper half-plane as usual and the arcs of $M_0$ are drawn in the lower half-plane, reflected about the number axis.  If the openers of $M$ are $o_1 < \cdots < o_n$, we define
$$\Cyc(M, M_0) = \{k : o_k \text{ is a minimal vertex in a cycle in the graph } (M, M_0)\}.$$ Figure~\ref{Cycles} shows the calculation of $\cyc$ and $\Cyc$ for all matchings of type \includegraphics[height=0.2cm]{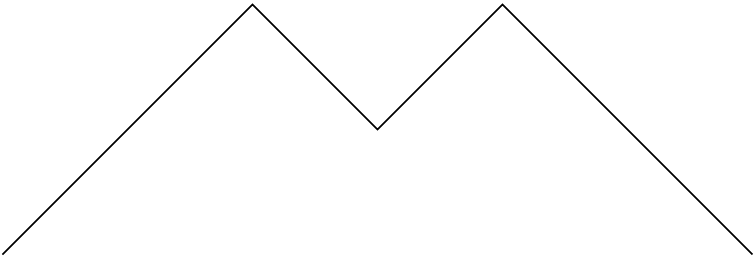} with respect to the nonnesting matching of that type.

\begin{figure}[h]
\begin{center}
\includegraphics[height=2cm]{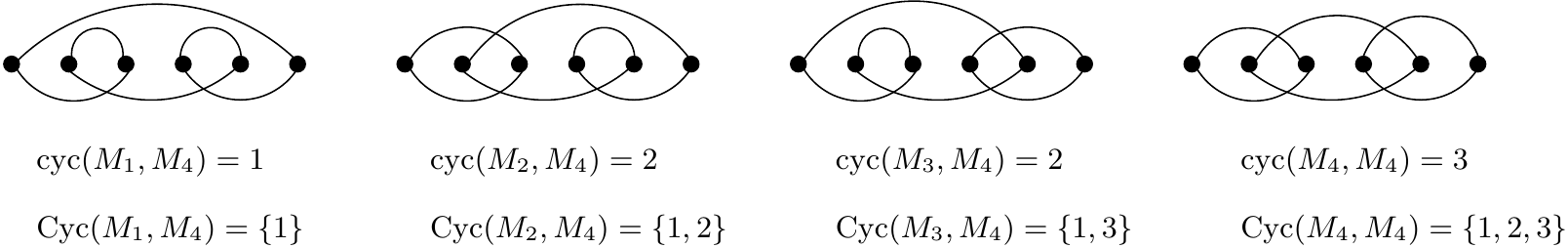}
\end{center}
\caption{Counting cycles in matchings.}
\label{Cycles}
\end{figure}

For  $M, M_0 \in \mathcal{M}_n(D)$, we define the sorting index of $M$ with respect to $M_0$, denoted by $\sor(M, M_0)$, in the following way. Let $o_1 < o_2 < \dots < o_n$ be the openers in $M$ and $M_0$. We construct a sequence of matchings $M_n, M_{n-1}, \dots,  M_2, M_1$ as follows. First, set $M_n =M$. Then, if $M_k(o_k) = M_0(o_k)$, set $M_{k-1} = M_k$. Otherwise, set $M_{k-1}$ to be the matching obtained by replacing the edges $o_k \cdot M_k(o_k)$ and $ M_k(M_0(o_k)) \cdot  M_0(o_k)$ in the matching $M_k$ by the edges 
$o_k \cdot M_0(o_k)$ and $ M_k(M_0(o_k)) \cdot  M_k(o_k)$.  It follows from the definition that $M_1 =M_0$. In other words,  we gradually sort the matching $M$ by reconnecting the  openers to the closers  as ``prescribed'' by $M_0$. Note that when swapping of edges  takes place, it is always true that $M_k(M_0(o_k)) < o_k$ and therefore all the intermediary matchings  we get in the process are of  type $D$. Define 
\begin{equation*}
\sor_k(M, M_0) = \begin{cases}
|\{ c :  c > o_k, c \in [M_k(o_k), M_0(o_k)]  \text{ and } M_0(c) < o_k\}| , &\text { if }  M_k(o_k) \leq  M_0(o_k)\\
 |\{ c :  c > o_k, c \notin \left(M_0(o_k), M_k(o_k)\right)  \text{ and } M_0(c) < o_k\}|, &\text{ if }  M_0(o_k) < M_k(o_k) \end{cases}
 \end{equation*}
and
\begin{equation*} \sor(M, M_0) = \sum_{k=1}^n \sor_k(M, M_0). \end{equation*}

\begin{example} Figure~\ref{fsorting} shows the intermediate matchings that are obtained when  $M = M_6$  is sorted to  $M_0 = M_1$. So,

\begin{center}
\begin{tabular} {lll} $\sor_6(M, M_0) = |\{c_3, c_5, c_6\}|  = 3,$ &  $\sor_5(M, M_0) = |\{c_3, c_5\}| = 2,$ & $\sor_4(M, M_0) = |\{c_2, c_5\}| = 2,$ \\
$\sor_3(M, M_0) = | \emptyset | = 0,$ & $\sor_2(M, M_0) = |\{ c_5\}| = 1,$ & $\sor_1(M, M_0) = | \emptyset|  = 0,$
\end{tabular}
\end{center}
and $\sor(M, M_0) = 0 + 1+ 0 + 2 + 2 + 3 = 8$.
\end{example}

\begin{figure}[h]
\begin{center}
\includegraphics[height=6cm]{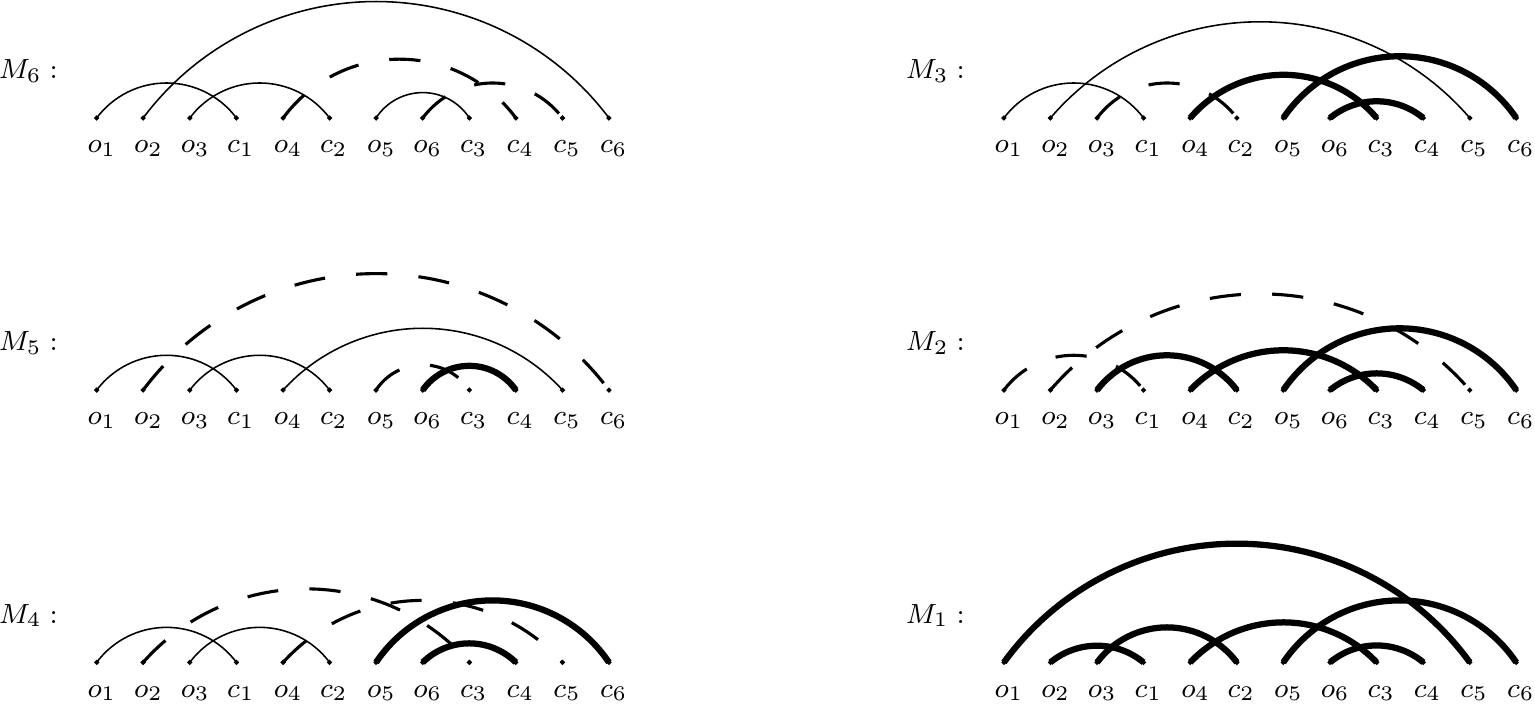}
\end{center}
\caption{Sorting of the matching $M = M_6$ to the matching $M_0 = M_1$. The dashed lines indicate arcs that are about to be swapped while the bold lines represent arcs that have been placed in correct position.}
\label{fsorting}
\end{figure}

\begin{theorem} \label{thmcycles} Let $D$ be a Dyck path with height sequence  $(h_1, \dots, h_n)$. For each $M_0 \in \mathcal{M}_n(D)$, there is a bijection \[\phi_1:  \{(w_1, w_2, \dots, w_n): 1 \leq w_i \leq h_i\} \rightarrow \mathcal{M}_n(D)\] which depends on $M_0$ such that
\begin{itemize}
\item[(a)] $ \sor(\phi_1(w_1, \dots, w_n), M_0) = \sum_{i=1}^n (w_i-1)$,
\item[(b)] $\Cyc(\phi_1(w_1, \dots, w_n), M_0) = \{k: w_k =1\}$.
\end{itemize}
Additionally, if $M_0$ is the unique nonnesting matching of type $D$, then
\begin{itemize}
\item[(c)] $\Short(\phi_1(w_1, \dots, w_n)) = \Rlminl(2-w_1, 3-w_2, \dots, n+1-w_n)$.
\end{itemize}
 \end{theorem}

\begin{proof} Fix $M_0 \in \mathcal{M}_n(D)$. We construct the bijection $\phi_1$ in the following way. Draw the matching $M_0$ with arcs in the lower half-plane. Suppose $o_1< \dots < o_n$ are the openers of $M_0$. To construct $M=\phi_1(w_1, \dots, w_n)$, we draw arcs in the upper half plane by connecting the openers from right to left to closers as follows. 

Suppose that the openers $o_n, o_{n-1}, \dots, o_{k+1}$ are already connected to a closer and denote the partial matching in the upper half-plane by $N_k$. To connect $o_k$, we consider all  the closers $c$ with the property $c > o_k$ and $M_0(c) \leq o_k$. There are exactly $h_k$ such closers, call them candidates for $o_k$. 

Let $c_{k_0}$ be the closer which is $w_k$-th on the list when all those $h_k$ candidates are listed starting from $M_0(o_k)$ and then going cyclically to left. If $c_{k_0}$ is not connected to an opener by an arc in the upper half-plane, draw the arc $o_k \cdot c_{k_0}$. Otherwise, there is a maximal path in the graph of the type: $c_{k_0}, N_k(c_{k_0}), M_0(N_k(c_{k_0})), N_k(M_0(N_k(c_{k_0}))), \dots, c^*$ which starts with $c_{k_0}$, follows arcs in $N_k$ and $M_0$ alternately and ends with a closer $c^*$ which has not been connected to an opener yet (see Figure~\ref{bijtwo}). Due to the order in which we have been drawing the arcs in the upper half-plane, all vertices in the aforementioned path are to the right of $o_k$. In particular, $c^*$ is to the right of $o_k$ and is not one of the candidates for $o_k$. Draw an arc in the upper half-plane connecting $o_k$ to $c^*$. After all openers are  connected in this manner, the resulting matching in the upper half-plane is $M=\phi_1(w_1, \dots, w_n)$.

\begin{figure}[h]
\begin{center}
\includegraphics[height=3cm]{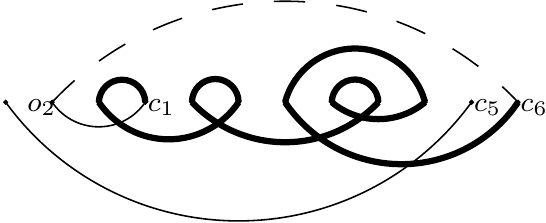}
\end{center}
\caption{The solid arcs in the top half-plane represent the partial matching $N_2$. The candidates for $o_2$ are $c_1$ and $c_5$. If $w_2 = 1$, $o_2$ will try to connect to $c_1$, but since it is already connected to an opener, we follow the bold path that starts with $c_1$ to reach $c^* = c_6$ and connect it to $o_2$.}
\label{bijtwo}
\end{figure}

Let $M_n =M, M_{n-1}, \dots, M_2, M_1=M_0$ be the intermediary sequence of matchings constructed when sorting $M$ to $M_0$. Then $M_k(o_k)$ is exactly the closer $c_{k_0}$ defined above. This means that $\sor_k(M, M_0) = w_k-1$ and therefore $\sor(M, M_0) = \sum_{k=1}^n (w_k-1)$. This property also gives us a way of finding the sequence $w_1, \dots, w_n)$ which corresponds to a given $M \in \mathcal{M}_n(D).$ Namely, $w_k = \sor_k(M, M_0)+1$.

To prove the second property of $\phi_1$, we analyze when connecting $o_k$ by an arc will close a cycle. There are two cases.
\begin{enumerate}
\item The closer $c_{k_0}$ which was $w_k$-th on the list of candidates for $o_k$ was not incident to an arc in the partial matching $N_k$ and we drew the arc $o_k \cdot c_{k_0}$. If $w_k=1$, then $c_{k_0} = M_0(o_k)$ and the arcs connecting $o_k$ and $c_{k_0}$ in the upper and lower half-planes close a cycle. Otherwise,  $M_0(c_{k_0}) < o_k$ and therefore $M_0(c_{k_0})$ is not incident to an arc in $N_k$ and the arc $o_k \cdot c_{k_0}$ will not close a cycle.

\item The closer $c_{k_0}$ which was $w_k$-th on the list of candidates for $o_k$ was incident to an arc in the partial matching $N_k$ and we drew the arc $o_k \cdot c^*$. If $w_k=1$, the path traced from $c_{k_0}$ to $c^*$, the arc $o_k \cdot c_{k_0}$ in $M_0$, and the newly added arc $o_k \cdot c^*$ form a cycle. Otherwise,  connecting $o_k$ to $c^*$ does not close a cycle since the opener  $M_0(c_{k_0})$ is in the same connected component of the graph $(M, M_0)$ as $o_k$, but is not connected to a closer yet, since $M_0(c_{k_0}) < o_k$. \end{enumerate}

We conclude that a cycle is closed exactly when $w_k = 1$ and therefore $$\Cyc(\phi_1(w_1, \dots, w_n), M_0) = \{k: w_k =1\}.$$

Finally, we prove the third property of $\phi_1$. If $M_0$ is a nonnesting matching, its edges are $o_k \cdot c_k$ where the openers and closers are indexed in ascending order.  Let $M = \phi_1(w_1, \dots, w_n)$. The following observations are helpful. When connecting $o_k$ in the construction of $M$, the first choice for $o_k$, i.e.,  the $w_k$-th candidate for $o_k$ is exactly $c_{k+1-w_k}$. Also, $M(o_k) \geq c_{k+1-w_k}$. Furthermore, if $c_k$ was not a candidate for $M(c_k)$, i.e. if the edge $c_k$ was chosen as a partner for $M(c_k)$ by following a path in the graph as described above, then $k \notin \Short(M)$. Namely the edge $M(c_{k_0}) \cdot c_{k_0}$, where $c_{k_0}$ was the first choice when the opener $M(c_k)$ was connected in the construction of $M$, is nested below it.

For a number $k \in [n]$ there are three possibilities:
\begin{enumerate}

\item $k \notin \{2-w_1, 3-w_2, \dots, n+1-w_n\}$

In this case, $c_k$ was not a first choice for any of the openers and therefore must have been connected to an opener by following a path in the graph $(M, M_0)$. It follows from the observation above that $k \notin \Short(M)$. 

\item $k \in \{2-w_1, 3-w_2, \dots, n+1-w_n\}$ and $k \in \Rlminl(2-w_1, 3-w_2, \dots, n+1-w_n)$

Then $c_k$ was a first choice for at least one opener. Let $o$ be the largest one. Then all openers to the right of $o$ got connected to a closer which is greater than $c_k$, so no edge is nested below $o \cdot c_k \in M$. Consequently, $k \in \Short(M)$.

\item $k \in \{2-w_1, 3-w_2, \dots, n+1-w_n\}$ but $k \notin \Rlminl(2-w_1, 3-w_2, \dots, n+1-w_n)$

In this case, let $m+1-w_{m}$ be the rightmost number in the sequence $(2-w_1, \dots, n+1-w_n)$ which is smaller than $k$. It is necessarily to the right of $k$ in this sequence and belongs to $\Rlminl(2-w_1,  \dots, n+1-w_n)$. This implies that the edge $o_m \cdot c_{m+1-w_{m}}$ is in $M$, while $M(o_l) > c_k$ for all $l >m$. So, $M(c_k) < o_m$ and therefore the edge $o_m \cdot c_{m+1-w_{m}}$ is nested below $M(c_k) \cdot c_k$, which means that $k \notin \Rlminl(2-w_1, 3-w_2, \dots, n+1-w_n)$.

\end{enumerate}

\end{proof}

 \begin{corollary} \label{cora} Let $M_0 \in \mathcal{M}_n(D)$ and let $(h_1, \dots, h_n)$ be the height sequence of $D$. Then

\begin{equation} \label{eqhm} \sum_{M \in \mathcal{M}_{n}(D)} {q^{\sor(M, M_0)}  \prod_{i \in \Cyc(M, M_0)} t_i } =   \prod_{k=1}^n (t_k + q + \cdots + q^{h_{k}-1}). \end{equation}
 
\end{corollary}

Combining Theorem~\ref{thmbasic} and Corollary~\ref{cora} we get the following corollary.

\begin{corollary} Let $M_0 \in \mathcal{M}_n(D)$ and let $(h_1, \dots, h_n)$ be the height sequence of $D$. Then  \[ \sum_{M \in \mathcal{M}_{n}(D)} {q^{\sor(M, M_0)}  \prod_{i \in \Cyc(M, M_0)} t_i } = \sum_{M \in \mathcal{M}_{n}(D)} {q^{\nn(M)}  \prod_{i \in \Long(M)} t_i }.\] 
\end{corollary}

\begin{corollary} \label{lastcorcycles} If $M_0$ is the unique nonnesting matching of type $D$ then the multisets
 $$\{(\sor(M, M_0), \Cyc(M, M_0), \Short(M)) : M \in \mathcal{M}_n(D)\}$$ and $$\{(\nn(M), \Long(M), \Short(M)) : M \in \mathcal{M}_n(D)\}$$ are equal.
\end{corollary}
\begin{proof} Follows from Proposition~\ref{propone} and Theorem~\ref{thmcycles}.
\end{proof}

 \subsection{Connections with restricted permutations}
 
For a fixed $n$, let $\bf{r}$ denote the non-decreasing sequence of integers $1 \leq r_1 \leq r_2 \leq \cdots \leq r_n \leq n$.  Let 
  
 \[ S_{\bf{r}} = \{ \sigma \in S_n : \sigma(k) \leq r_k, 1 \leq k \leq n\}. \]
 
Note that $S_{\bf{r}} \neq \emptyset$ precisely when $r_k \geq k$, for all $k$, so we will consider only the sequences that satisfy this condition without explicitly mentioning it. Let $D(\bf{r})$ be the unique Dyck path whose $k$-th fall is preceded by  exactly $r_k$ rises. Consider the following bijection $f_{\bf{r}}: S_{\bf{r}} \rightarrow \mathcal{M}_n(D(\bf{r}))$. If $\sigma \in S_{\bf{r}}$, then $f_{\bf{r}}(\sigma)$ is the matching in $\mathcal{M}_n(D(\bf{r}))$ with edges $o_{\sigma(k)} \cdot c_k$, where  $o_1 < \cdots < o_n$ are the openers and $ c_1 < \cdots < c_n$ are the closers. It is not difficult to see that $f_{\bf{r}}$ is well defined and that it is a bijection.

Two arcs $o_{\sigma(j)} \cdot c_j$ and $o_{\sigma(k)} \cdot c_k$ in $f_{\bf{r}}(\sigma)$ with $j <k$  form a nesting if and only if  $\sigma(j) > \sigma(k)$. So,  $\nn(f_{\bf{r}}(\sigma)) = \inv(\sigma)$. Moreover, $\sigma(j) \in \Rlminl(\sigma)$ if and only if $\sigma(j)$ does not form an inversion with a $\sigma(k)$ for any $k > j$, which means if and only if $o_{\sigma(j)} \cdot c_j$ is not nested within anything in $f_{\bf{r}}(\sigma)$, i.e.,  $\sigma(j) \in \Long(f_{\bf{r}}(\sigma))$. From Theorem~\ref{thmbasic} we get the following corollary.

\begin{corollary} Let $\bf{r}$  be a non-decreasing sequence of integers $1 \leq r_1 \leq r_2 \leq \cdots \leq r_n \leq n$ with $r_k \geq k$, for all $k$. Then
\[ \sum_{\sigma \in S_{\bf{r}} } q^{\inv(\sigma)} \prod_{i \in \Rlminl(\sigma)} t_i = \prod_{k=1}^n (t_k    +  q   +  q^2 + \cdots +  q^{h_k-1})\] where $(h_1, \dots, h_n)$ is the height sequence of $D(\bf{r})$. In particular,
\[ \sum_{\sigma \in S_{\bf{r}} } q^{\inv(\sigma)} t^{\mathrm{rlminl}(\sigma)} = \prod_{k=1}^n (t    +  q   +  q^2 + \cdots +  q^{r_k-k}).\]

\end{corollary}

 \begin{proof} The first result  follows directly from the discussion above and Theorem~\ref{thmbasic}. For the second equality, note that the height sequence $(h_1, \dots, h_n)$ of the Dyck path $D(\bf{r})$ is a permutation of the sequence of the heights of the falls in $D(\bf{r})$, where the height of a fall is the $y$-coordinate of the higher end of the corresponding $(1,-1)$ step. The height of the $k$-th fall  is easily seen to be $r_k - k +1$. 
 \end{proof}
 
 In particular, when $r_1 = r_2 = \cdots = r_n =n$, we have $S_{\bf{r}} = S_n$. The height sequence of  $D(\bf{r})$ is $(1, 2, \dots, n)$ and we recover the result of Bj\"{o}rner and Wachs about the distribution of $(\inv, \mathrm{Rlmin})$ given in~\eqref{BW}. 
 
If $M_0 \in \mathcal{M}(D({\bf{r}}))$ the sorting index $\sor(\; \cdot \;, M_0)$ induces a permutation statistic on $S_{\bf{r}}$. Namely, if $\sigma, \sigma_0 \in S_{\bf{r}}$, define $$\sor_{\bf{r}}(\sigma, \sigma_0) =  \sor(f_{\bf{r}}^{-1}(\sigma), f_{\bf{r}}^{-1}(\sigma_0)).$$ 

Equivalently, the statistic $\sor_{\bf{r}}(\sigma, \sigma_0)$ on $S_{\bf{r}}$ can be defined directly via a sorting algorithm similar to Straight Selection Sort. Namely, permute the elements in $\sigma \in S_{\bf{r}}$ by applying transpositions which place the largest element $n$ in position $\sigma_0^{-1}(n)$, then the element $n-1$ in position $\sigma_0^{-1}(n-1)$, etc. Let $\sigma_n = \sigma, \sigma_{n-1}, \dots, \sigma_1 = \sigma_0$, be the sequence of permutations obtained in this way. Specifically, $\sigma_k^{-1}(i) = \sigma_0^{-1}(i)$ for $i >k$, and $\sigma_{k-1}$ is obtained by swapping $k$ and $\sigma_k(\sigma_0^{-1}(k))$ in $\sigma_{k}$. 
 
 Let $l = \sigma_k^{-1}(k)$ and $m = \sigma_0^{-1}(k)$.  Define
 
 \begin{equation} \label{akdef} a_k = \begin{cases} |\{ i : l \leq i \leq m, \sigma_0(i) < k \}|, &l<m \\
 0, & l=m \\
|\{i: r_i \geq k, i \notin (m,l), \sigma_0(i) <k\}|, &l>m. \end{cases} \end{equation} Then $$\sor_{\bf{r}}(\sigma, \sigma_0) = \sum_{k=1}^n a_k.$$

Note that, $\sor_{\bf{r}}(\sigma, \sigma_0)$ in general depends on $\bf{r}$. However, the case when $\sigma_0$ is the identity permutation is an exception.
 
 \begin{lemma} \label{permlemma1} Let $\bf{r}$  be a non-decreasing sequence of integers $1 \leq r_1 \leq r_2 \leq \cdots \leq r_n \leq n$ with $r_k \geq k$, for all $k$. Let $\sigma \in S_{\bf{r}}$.  Then $$\sor_{\bf{r}}(\sigma, \text{\bf{id}}) = \sor(\sigma).$$
 \end{lemma}
 \begin{proof}  First note that the case $l > m$ in~\eqref{akdef} cannot occur. Namely, in the case when $\sigma_0 = \text{\bf{id}}$, we have $m=k$ and if $l>k$, $\sigma_k^{-1}(l) = \sigma_0^{-1}(l) = l.$ This contradicts $l = \sigma_k^{-1}(k)$. Therefore, the definition of $a_k$ simplifies to
\[  a_k = | \{ i : l \leq i < k \}|. \] This is precisely the ``distance'' that $k$ travels when being placed in its correct position with the Straight Selection Sort algorithm.
  \end{proof}

 \begin{corollary} Let $\bf{r}$  be a non-decreasing sequence of integers $1 \leq r_1 \leq r_2 \leq \cdots \leq r_n \leq n$ with $r_k \geq k$, for all $k$. Let $\sigma_0 \in S_{\bf{r}}$. Then
  \begin{equation} \label{permcor} \sum_{\sigma \in S_{\bf{r}} } {q^{\sor_{\bf{r}}(\sigma, \sigma_0)}  \prod_{i \in \Cyc(\sigma \sigma_0^{-1})} t_i } =   \prod_{i=1}^n (t_i + q + \cdots + q^{h_{i}-1}),\end{equation} where $(h_1, \dots, h_n)$ is the height sequence of $D(\bf{r})$ and $\Cyc(\sigma)$ is the set of the minimal elements in the cycles of $\sigma$. In particular, 
   \begin{equation} \label{permcortwo}\sum_{\sigma \in S_{\bf{r}} } {q^{\sor(\sigma)}  \prod_{i \in \Cyc(\sigma)} t_i } =   \prod_{i=1}^n (t_k + q + \cdots + q^{h_{k}-1})\end{equation} and
   \begin{equation} \sum_{\sigma \in S_{\bf{r}} } q^{\sor(\sigma)} t^{\cyc(\sigma)} = \sum_{\sigma \in S_{\bf{r}} } q^{\inv(\sigma)} t^{\mathrm{rlminl}(\sigma)}  \end{equation}
  
 \end{corollary}
 
 \begin{proof} Let $f_{\bf{r}} (\sigma_0) = M_0$ and $f_{\bf{r}} (\sigma) = M$. The cycle $ k \rightarrow \sigma_0\sigma^{-1}(k) \rightarrow \cdots \rightarrow (\sigma_0\sigma^{-1})^s(k) =k$ of the permutation $\sigma_0\sigma^{-1}$ corresponds to the cycle $o_k \; \includegraphics[height=0.2cm]{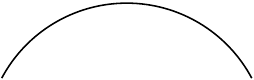} \; M(o_k)  \; \includegraphics[height=0.2cm]{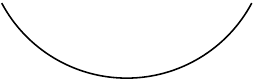} \; M_0(M(o_k)) \;  \includegraphics[height=0.2cm]{arc-crop.pdf} \; \cdots  \; \includegraphics[height=0.2cm]{arcreflected-crop.pdf} \; o_k$ in the graph $(M, M_0)$. So, $k \in \Cyc(\sigma_0\sigma^{-1})$ if and only if $k \in \Cyc(M, M_0)$. Now,~\eqref{permcor} follows from~\eqref{eqhm} and the fact that the cycles of $\sigma\sigma_0^{-1}$ are equal to  the cycles of $\sigma_0\sigma^{-1}$ reversed. Since $\text{\bf{id}} \in S_{\bf{r}}$ for every sequence $\bf{r}$, we get~\eqref{permcortwo} as a corollary of Lemma~\ref{permlemma1}. 
 \end{proof}
 
 Let $\Lrmaxp(\sigma)$ denote the set of left-to-right maximum places in the permutation $\sigma$, i.e, 
 \[ \Lrmaxp(\sigma) = \{ k: \sigma(k) > \sigma (j) \text{ for all } j<k \}.\]  From Corollary~\ref{lastcorcycles} we get the following result for restricted permutations. 
 \begin{corollary} The triples $(\inv, \Rlminl, \Lrmaxp)$ and $(\sor, \Cyc, \Lrmaxp)$ are equidistributed on $S_{\bf{r}}$. That is, the multisets $$\{ (\inv(\sigma), \Rlminl(\sigma), \Lrmaxp(\sigma)): \sigma \in S_{\bf{r}} \}$$ and $$\{ (\sor(\sigma), \Cyc(\sigma), \Lrmaxp(\sigma)) : \sigma \in S_{\bf{r}} \}$$ are equal.
 
 \end{corollary} The equidistribution of the pairs $(\Rlminl, \Lrmaxp)$ and $(\Cyc, \Lrmaxp)$ on $S_{\bf{r}}$ for the special case when  the corresponding Dyck path $D(\bf{r})$ is of the form $u^{k_1}d^{k_1}u^{k_2}d^{k_2} \cdots u^{k_s}d^{k_s}$ was shown by Foata and Han~\cite{FH} .

 \begin{corollary} Let  $\sigma_0 \in S_{\bf{r}}$. Then
  \begin{equation} \label{lasthm} \sum_{\sigma \in S_{\bf{r}} }    t^{\cyc(\sigma \sigma_0^{-1})}  = \prod_{k=1}^{n} (t + r_k -k).\end{equation} In particular, the left-hand side of~\eqref{lasthm} does not depend on $\sigma_0$. 
 \end{corollary}
 
 We remark that the sets  $\{ \sigma \sigma_0^{-1} : \sigma \in S_{\bf{r}} \}$ and $S_{\bf{r}}$ are in general  not equal. For example, let $\sigma_0 = 143265 \in S_{[4, 4, 4, 6, 6, 6]}$. Then $\sigma = 231546  \in S_{[4, 4, 4, 6, 6, 6]}$ but $\sigma \sigma_0^{-1} = 251364 \notin  S_{[4, 4, 4, 6, 6, 6]}$.

The polynomial $\prod_{k=1}^{n} (t + r_k -k)$ is well-known in rook theory. It is  equal~\cite{G} to the polynomial $$\sum_{k=0}^n r_{n-k} (t-1) (t-2) \cdots (t-k)$$ where $r_k$ is the number of placements of $k$ non-atacking rooks on a Ferrers board with rows of length $r_1, r_2, \dots, r_n$.


\section{Bicolored matchings}
\label{bimat}

In this section we consider statistics on the set $\mathcal{M}_n^{(2)}$ of bicolored matchings on $[2n]$,  whose $n$ edges are colored with one of two colors: red or blue.  

\subsection{Bicolored crossings and nestings}

Bicolored matchings have four types of crossings, depending on the color of the right and the left edge that form the crossing, as well as four types of nestings and four types of alignments. Let
$\cc_{*r}(M)$ be the number of crossings in $M$ in which the right edge is red, regardless of the color of the left edge, and analogously define the numbers $\cc_{*b}(M)$, $\nn_{*r}(M)$, $\nn_{*b}(M)$, $\al_{*r}(M)$, $\al_{*b}(M)$. Additionally, let $\mathrm{b}(M)$ denote the total number of blue edges in $M$, and let $\mathrm{longr}(M)$ denote the number of long red edges -- red edges in $M$ that are not nested within any other edge, while
\[\mathrm{Longr}(M) = \{k: o_k \cdot M(o_k) \text{ is a large red edge}\}.\] 
The generating function of these refined statistics for bicolored matchings of type $D$ 
\[ P_D({\bf{q}},p, {\bf{t}}) =\sum_{M \in  \mathcal{M}_n^{(2)}(D)} q_1^{\nn_{*r}(M)}q_2^{\nn_{*b}(M)}q_3^{\cc_{*r}(M)}q_4^{\cc_{*b}(M)}q_5^{\al_{*r}(M)}q_6^{\al_{*b}(M)} p^{\mathrm{b}(M)} \prod_{i \in \mathrm{Longr}(M)} t_i  \]  is given by the following theorem. In the proof we will use the set $\mathcal{WD}^{(2)}_n$ of all (partially) bicolored weighted Dyck paths whose rises are colored red or blue. The elements in $\mathcal{WD}^{(2)}_n$ can be written as triples $(D, (w_1, \dots, w_n), (\epsilon_1, \dots, \epsilon_n))$ where $D$ is a Dyck path of semilength $n$, $w_i \in \mathbb{Z}$ with $ 1\leq w_i \leq h_i$, where $(h_1, \dots, h_n)$ is the height sequence of $D$, and $\epsilon_i \in \{0,1\}$, for $1 \leq i \leq n$. Here we are using $\epsilon_i = 0$ to represent a red rise and $\epsilon_i = 1$ to represent a blue rise.

\begin{theorem} \label{thmsigned}
Let $D$ be a Dyck path with height sequence $(h_1, \dots, h_n)$. 
Then
\[ P_D({\bf{q}},p, {\bf{t}})= \prod_{i=1}^{n} \sum_{k=1}^{h_i}(q_1^{k-1} q_3^{h_i-k} q_5^{i-h_i} t_i^{\delta_{k,1}} + q_2^{h_i-k} q_4^{k-1} q_6^{i-h_i} p  ),\]
where $\delta_{i,j}$ is the Kronecker delta function.
\end{theorem}
\begin{proof} 
To find $P_D(q_1,q_2,q_3,q_4,q_5,q_6,p,t)$ we will describe an appropriate bijection  $\varphi_2$ from $\mathcal{WD}^{2}_n$ to $\mathcal{M}_n^{(2)}$. Let $(D, (w_1, \dots, w_n), (\epsilon_1, \dots, \epsilon_n)) \in \mathcal{WD}^{(2)}_n$. The corresponding matching $M$ has type $D$ and is constructed in the following way. 
The openers $o_1 < \cdots < o_n$ of $M$ are connected to closers from right to left starting with $o_n$. If $\epsilon_k=0$,  connect the opener $o_k$ to the $w_i $-th available closer  from right to left and color the edge red. If $\epsilon_k =1$, connect the opener to the $(h_i - w_i +1)$-st available closer to the right of $o_k$ counted from right to left and color the edge blue. Then if $\epsilon_k=0$ the corresponding red edge $o_k \cdot M(o_k)$ will be a right edge in $w_i-1$ nestings,  $h_i-w_i$ crossings, and  $i -h_i$ alignments. It will be a long edge if and only if $w_i = 1$. Similarly, if $\epsilon_k=1$, then the blue edge $o_k \cdot M(o_k)$ will be a right edge in $h_i-w_i$ nestings, $w_i-1$ crossings, and $i -h_i$ alignments. It will not be a long edge but it will contribute to $\mathrm{b}(M)$. The theorem follows by summing over all possible colorings and weightings of the path $D$.
\end{proof}

This theorem gives the generating function for statistics that can be defined in terms of nestings,  crossing,  and alignments with right edges of specified color. In particular, define
\[\mathrm{mix}(M) = \nn(M) + 2\cc_{\ast b}(M) + 2\al_{\ast b}(M) + \mathrm{b}(M).\]
\begin{corollary} \label{corb} Let $D$ be a Dyck path with height sequence $(h_1, \dots, h_n)$. 
\[ \sum_{M \in  \mathcal{M}_n^{(2)}(D)} q^{\mathrm{mix}(M)} \prod_{i \in \mathrm{Longr}(M)} t_i = \prod_{k=1}^n\left(t_k+ q [h_k -1]_q + q^{2k-h_k} [h_k]_q\right). \]
\end{corollary}

\begin{proof} It follows from Theorem~\ref{thmsigned} by setting $q_1=q_2=q$, $q_4 = q_6 = q^2$, $q_3 = q_5 =1$, and $p=q$.
\end{proof}


\subsection{Sorting index and cycles for bicolored matchings}

For a vertex $v$ in a bicolored matching $M$, denote 
\begin{equation*}
\col( v, M) = \begin{cases}
0, & \text{ if the edge in $M$ incident with $v$ is red}\\
1,  & \text{ if the edge in $M$ incident with $v$ is blue}.
\end{cases}
\end{equation*}

Let $M_0$ be a matching with only red edges of type $D$. For $M \in \mathcal{M}_n^{(2)}(D)$ we define $\sor(M, M_0)$ similarly as for monochromatic matchings. Namely, let $o_1 < o_2 < \dots < o_n$ be the openers in $M$ (and consequently in $M_0$ as well). Define the sequence of bicolored matchings $$M = M_n, M_{n-1}, \dots,  M_2, M_1$$ as follows. Supose $M_n, \dots, M_k$ are defined for some $k \leq n$. Then, if $M_k(o_k) = M_0(o_k)$ the matching $M_{k-1} $ has the same edges as $M_k$ with colors
\begin{equation*} 
\col(o_i, M_{k-1}) = 
\begin{cases}
\col(o_i, M_{k}), & \text {if $i \neq k$}\\
0,& \text {if $i= k$}.
\end{cases} 
\end{equation*}
Otherwise, $M_{k-1}$ is the bicolored matching obtained by replacing the two edges $o_k \cdot M_k(o_k)$ and $ M_k(M_0(o_k)) \cdot  M_0(o_k)$ in the matching $M_k$ by the edges 
$o_k \cdot M_0(o_k)$ and $ M_k(M_0(o_k)) \cdot  M_k(o_k)$ and setting their colors to be
\begin{equation*} 
\col(o_i, M_{k-1}) = 
\begin{cases}
\col(o_i, M_{k}), & \text {if $o_i \neq o_k$ and $ o_i \neq M_k(M_0(o_k))$}\\
\col(o_i, M_{k}) + \col(o_k, M_{k}) (\text{mod }2),& \text {otherwise}.
\end{cases} 
\end{equation*}
In other words,  the sequence of matchings $M = M_n, M_{n-1}, \dots,  M_2, M_1$ is the one we get when we gradually sort the matching $M$ from right to left by connecting the openers to the closers as prescribed by $M_0$ and recoloring edges depending on the color of the edge which is currently being ``processed''. Note that $\col(o_k, M_i)=0$ for $i<k$ and so the final matching that we get after $n$ steps is indeed the desired $M_0$.  
 
The sorting index $\sor(M, M_0)$ is now defined in the following way. \\
\noindent If $\col(o_k, M_k) = 0$ define
\begin{equation*}
\sor_k(M, M_0) = \begin{cases}
|\{ c :  c > o_k, c \in [M_k(o_k), M_0(o_k)]  \text{ and } M_0(c) < o_k\}| , &\text { if }  M_k(o_k) \leq M_0(o_k)\\
 |\{ c :  c > o_k, c \notin (M_0(o_k), M_k(o_k))  \text{ and } M_0(c) < o_k\}|, &\text{ if }  M_0(o_k) < M_k(o_k) \end{cases}
 \end{equation*}
while if $\col(o_k, M_k) = 1$ define 
\begin{equation*}
\sor_k(M, M_0) = \begin{cases}
2k-1 - |\{ c :  c > o_k, c \in [M_k(o_k), M_0(o_k)], M_0(c) < o_k\}|, &\\
\hspace{ 10cm} \text { if }  M_k(o_k) \leq M_0(o_k) &\\
2k-1- |\{ c :  c > o_k, c \notin (M_0(o_k), M_k(o_k)) , M_0(c) < o_k\}|, &\\
 \hspace{ 10cm} \text{ if }  M_0(o_k) < M_k(o_k).
 \end{cases}
 \end{equation*}
Then $\sor(M, M_0)$ is defined as 
\begin{equation*} \sor(M, M_0) = \sum_{k=1}^n \sor_k(M, M_0). \end{equation*}
In particular, if $M$ has only red edges, $\sor(M, M_0)$ is equal to $\sor(M, M_0)$ defined for monochromatic matchings in Section~\ref{mat}.

Similarly to the monochromatic case, we define $\cyc(M, M_0)$ to  be the number of cycles in the graph $(M, M_0)$. However, in this case we can distinguish between two types of cycles: $\cyc_0(M, M_0)$ will denote the number of cycles with even number of blue edges, while $\cyc_1(M, M_0)$ will denote the number of cycles with odd number of blue edges. We will denote by $\Cyc_0(M, M_0)$ and $\Cyc_1(M, M_0)$ the indices of the minimal openers in the respective sets.

\begin{theorem}
 Let $M_0 \in \mathcal{M}_n^{(2)}(D)$ be a  matching with all red edges of type $D$ and let  $(h_1, \dots, h_n)$ be the height sequence of $D$. There is a bijection which depends on $M_0$
 
 \[ \phi_2: \{((w_1, \dots, w_n), (\epsilon_1, \dots, \epsilon_n)): w_i \in \mathbb{Z}, 1 \leq w_i \leq h_i, \epsilon_i \in \{0,1\}\} \rightarrow M_n^{(2)}(D)\] such that the matching $M = \phi_2((w_1, \dots, w_n), (\epsilon_1, \dots, \epsilon_n))$ has the following properties:
 \begin{equation} \label{sorthmb} \sor(M, M_0) =\sum_{k=1}^n (w_k + \epsilon_k (2k-h_k) -1), \end{equation}  and 
 \begin{equation} \cyc(M, M_0) = |\{ k: w_k= 1+ \epsilon_k (h_k  -1)\}|. \end{equation}
 Moreover, $\Cyc_0(M, M_0)  = \{k: (w_k, \epsilon_k) = (1,0)\}$ and $\Cyc_1(M, M_0) = \{k: (w_k, \epsilon_k) = (h_k, 1) \}$.\end{theorem}

\begin{proof} We construct $\phi_2$  in the following way. Draw the matching $M_0$ with red arcs in the lower half-plane. Let $o_1< \dots < o_n$ be the vertices that are to be openers in $M=\varphi((w_1, \dots, w_n), (\epsilon_1, \dots, \epsilon_n))$ as determined by the type $D$. We draw arcs in the upper half plane by connecting the openers from right to left to closers as follows. 

Suppose that the openers $o_n, o_{n-1}, \dots, o_{k+1}$ are already connected to a closer and denote the partial matching in the upper half-plane by $N_k$. In particular, $N_n$ is the empty matching. To connect $o_k$, we consider all  the closers $c$ with the property $c > o_k$ and $M_0(c) \leq o_k$. Note that there are exactly $h_k $ such closers, call them candidates for $o_k$.

 If $\epsilon_k=0$, let $c_{k_0}$ be the closer which is $w_k $-th on the list when all those $h_k$ candidates are listed starting from $M_0(o_k)$ and then going cyclically to the left. Otherwise, let  $c_{k_0}$ be the closer which is $(h_k-w_k +1)$-st on that list. If $c_{k_0}$ is not connected to an opener by an arc in the upper half-plane, draw the arc $o_k \cdot c_{k_0}$ with color $\col(o_k, N_{k-1})=\epsilon_k$. Otherwise, there is a maximal path in the graph of the type \begin{equation} \label{path} c_{k_0}, N_k(c_{k_0}), M_0(N_k(c_{k_0})), N_k(M_0(N_k(c_{k_0}))), \dots, c^* \end{equation} which starts with $c_{k_0}$, follows arcs in $N_k$ and $M_0$ alternately and ends with a closer $c^*$ which has not been connected to an opener yet. Note that, due to the order in which the arcs in the upper half-plane are drawn, all vertices in the aforementioned path are to the right of $o_k$. In particular, $c^*$ is to the right of $o_k$ and is not one of the candidates for $o_k$. Draw an arc in the upper half-plane connecting $o_k$ to $c^*$ and set its color to be
 \begin{equation} \label{bicolcyc} \col(o_k, N_{k-1}) =  \epsilon_k +  \text{the sum of the colors of the edges in that path } (\text{mod } 2).\end{equation} 
When all the openers are connected in this manner, the resulting matching in the upper half-plane is $M=\phi_2((w_1, \dots, w_n), (\epsilon_1, \dots \epsilon_n))$.

Let $M_n =M, M_{n-1}, \dots, M_2, M_1$ be the sequence of intermediary matchings constructed when $M$ is sorted to $M_0$. Then $M_k(o_k)$ is exactly the closer $c_{k_0}$ defined above and we claim that $\col(o_k, M_k) =\epsilon_k$. To prove this claim, consider the following graph $G$ which represents the effects of color changing of arcs. The vertices of $G$ are $\{1, 2, \dots, n\}$.  Two vertices $l <k$ are connected by an edge if the $M_k(M_0(o_k)))= o_l$. This means that $$\col(o_l, M_{k-1}) = \col(o_l, M_k) + \col(o_k, M_k).$$ Let $G_k$ be the induced subgraph of $G$ on the vertices $\{k, k+1, \dots, n\}$. One can prove by induction that when connecting the opener $o_k$ in the construction of $\phi_2$, the arcs that are traced in the path~\eqref{path} are exactly the ones that have openers that are in the connected component of $k$ in $G_k$ and $\col(o_k, M_k) =\epsilon_k$.  

Now that we know how the sequence of matchings $M_n =M, M_{n-1}, \dots, M_2, M_1$ relates to $(w_1, \dots, w_n), (\epsilon_1, \dots \epsilon_n)$, it is not difficult to find $\sor_k(M, M_0)$. Suppose first that $M_k(o_k) \leq M_0(o_k)$ and consider the closers in the set $$ \{ c :  c > o_k, c \in [M_k(o_k), M_0(o_k)]  \text{ and } M_0(c) \leq o_k\}.$$ If  $\epsilon_k=0$, the elements in this set are exactly the first $w_k$ candidates for $o_k$ in the construction of $M$. If $\epsilon_k=1$, this set contains the first $h_k - w_k +1$ candidates for $o_k$. So, in this case \begin{equation} \label{rtf}  \sor_k(M, M_0) = \begin{cases} w_k-1,\; &\text{ if } \epsilon_k=0 \\
2k-1-h_k+w_k, \; &\text{ if } \epsilon_k=1. \end{cases} \end{equation}The case $M_0(o_k) < M_k(o_k)$ is similar and we get again~\eqref{rtf}. This proves~\eqref{sorthmb}.

We note that the inverse map $\phi_2^{-1}$ is not difficult to construct. To recover $w_k$ and $\epsilon_k$ that correspond to a given matching $M$, sort $M$ to $M_0$. If  $M_n=M, M_{n-1}, \dots, M_1$ is the  sequence of intermediary matchings obtained in the process, set $\epsilon_k = \col(o_k, M_k)$ and $$w_k = \sor_k(M, M_0) +1 - \col(o_k, M_k) (2k - h_k ).$$

Similarly as in the case of monochromatic matchings,  a cycle in the graph $(M, M_0)$ is closed exactly when $c_{k_0} = M_0(o_k)$, which means when $w_k = 1$ and $\epsilon_k=0$ or when $ w_k = h_k$ and $\epsilon_k = 1$. This proves the second property of $\phi_2$. Moreover, it follows from~\eqref{bicolcyc} that if the edge $o_k \cdot M(o_k)$ is the edge with the smallest opener in its own cycle (i.e. the edge that closes a cycle when $\phi_2((w_1, \dots, w_n), (\epsilon_1, \dots, \epsilon_n))$ is constructed),  then
\[\epsilon_k =  \text{the sum of the colors of the edges in that cycle} (\mathrm{mod} 2). \] Hence the minimal openers of the cycles with even number of blue edges correspond to the pairs $(1,0)$ among $(w_1, \epsilon_1), (w_2, \epsilon_2), \dots, (w_n, \epsilon_n)$ while the minimal openers of the cycles with odd number of blue edges correspond to the  pairs $(w_k, \epsilon_k) = (h_k, 1)$  in that list.
\end{proof}
From the properties of the map $\phi_2$, we get the following generating function.
\begin{corollary} Let $(h_1, \dots, h_n)$ be the height sequence of  the Dyck path $D$ and let $M_0 \in \mathcal{M}_n^{(2)}(D)$ be a matching with all red edges.  Then

\begin{equation} \sum_{M \in \mathcal{M}_{n}^{(2)}(D)} q^{\sor(M, M_0)}  \prod_{i \in \Cyc_0(M, M_0)}t_i \prod_{i \in \Cyc_1(M, M_0)}s_i  =   \prod_{k=1}^n \left(t_k + (q +q^{2k-h_k})[h_k-1]_q+ s_kq^{2k-1} \right) \end{equation}

 \end{corollary}

 \begin{corollary} Let $(h_1, \dots, h_n)$ be the height sequence of  the Dyck path $D$ and let $M_0 \in \mathcal{M}_n^{(2)}(D)$ be a matching with all red edges.  Then
 \[ \sum_{M \in \mathcal{M}_{n}^{(2)}(D)} q^{\sor(M, M_0)}  \prod_{i \in \Cyc_0(M, M_0)}t_i = \sum_{M \in \mathcal{M}_{n}^{(2)}(D)} q^{\mathrm{mix}(M)}  \prod_{i \in \mathrm{Longr}(M)}t_i .\]
 \end{corollary}





\subsection{Connections with restricted signed permutations}


Petersen~\cite{Petersen} defined a sorting index for signed permutations. Every signed permutation $\sigma \in B_n$ can be uniquely written as a product $$ \sigma = (i_1 j_1)(i_2 j_2)\cdots (i_k j_k)$$ of transpositions such that $i_s < j_s$ for $1 \leq s \leq k$ and $0 < j_1 < \cdots < j_k$. Here the transposition $(i j)$ means to swap both $i$ with $j$ and $\bar{i}$ with  $\bar{j}$ (provided $ i \neq \bar{j}$). The type $B_n$ sorting index  is defined to be
\[ \sor(\sigma) = \sum_{s=1}^k (j_s - i_s - \chi(i_s < 0)).\] 

As before, the sorting index keeps track of the total distance the elements in $\sigma$ move when $\sigma$ is sorted using a ``type $B$'' Straight Selection Sort algorithm in which, using a transposition,  the largest number is moved to its proper place, then the second largest, and so on. For example, the steps for sorting $\sigma = \bar{5} 1 3 \bar{4} \bar{2}$ are
\[ 24\bar{3}\bar{1}{\bf{5}} \; \bar{5} 1 3 \bar{4} \bar{2} \xrightarrow{(\bar{1}{5})} \bar{5}{\bf{4}}\bar{3}\bar{1} \bar{2} \; 2 1 3 \bar{4} 5  \xrightarrow{(\bar{4} 4)}   \bar{5} \bar{4}  \bar{3}\bar{1} \bar{2} \; {\bf{2}} 1 3 4 5  \xrightarrow{(1 2)} \bar{5} \bar{4}  \bar{3}\bar{2} \bar{1} \; 1 2 3 4 5  \]
and therefore $ \sigma  =  (1 2)  (\bar{4} 4)(\bar{1}{5})$ and
$\sor(\sigma) = (2-1) + (4-(-4)-1) + (5-(-1) -1) = 13$.

Let ${\bf{r}}$ be a nondecreasing sequence of positive integers $r_1 \leq r_2 \leq \cdots \leq r_n \leq n$ and let
$$B_{\bf{r}} = \{ \sigma \in B_n : \sigma(i) \leq r_i\}.$$ As before, only sequences  ${\bf{r}}$ for which $r_i \geq i$ are of interest as otherwise the set $B_{\bf{r}}$ is empty. There is a canonical bijection $$g_{\bf{r}}: B_{\bf{r}} \rightarrow \mathcal{M}^{(2)}_n(D({\bf{r})}) $$ that takes a permutation $\sigma \in B_{\bf{r}}$ to the matching with edges $o_{|\sigma(k)|} \cdot c_k$ colored red if $\sigma(k) > 0$ and colored blue if $\sigma(k) < 0$, for $k \in \{1, 2, \dots, n\}$.

As in the monochromatic case, for each $ \sigma_0 \in B_{\bf{r}}$ with $\sigma(k) > 0$ for $k >0$, this map  induces a statistic on $B_{\bf{r}}$, which we will denote $\sor_{\bf{r}}(\; \cdot \;,\sigma_0)$:
$$\sor_{\bf{r}}(\sigma, \sigma_0) = \sor(g_{\bf{r}}^{-1}(\sigma), g_{\bf{r}}^{-1}(\sigma_0) ).$$

The direct definition of  $\sor_{\bf{r}}(\sigma, \sigma_0)$ on $B_{\bf{r}}$  via a sorting algorithm similar to Straight Selection Sort is as follows. Permute the elements in $\sigma \in B_{\bf{r}}$ by applying transpositions which place the largest element $n$ in position $\sigma_0^{-1}(n)$, then the element $n-1$ in position $\sigma_0^{-1}(n-1)$, etc. Let $\sigma = \sigma_n, \sigma_{n-1}, \dots, \sigma_1$, be the sequence of intermediary permutations obtained in this way. In particular, $\sigma_k^{-1}(i) = \sigma_0^{-1}(i)$ for $|i| >k$, and $\sigma_{k-1}$ is obtained by swapping $k$ and $\sigma_k(\sigma_0^{-1}(k))$ in $\sigma_{k}$. 
 
 Let $l = \sigma_k^{-1}(k)$ and $m = \sigma_0^{-1}(k) >0$.  Define
 
 \begin{equation} \label{bkdef} b_k = \begin{cases} 0, &  l=m \\
  |\{ i :  l \leq i \leq m, \sigma_0(i) < k \}|, &0< l<m \\
|\{i: r_i \geq k, i \notin (m,l), \sigma_0(i) <k\}|, & l > m \\
 2k-1, & l=-m\\
 2k-1 -  |\{ i :  l \leq i \leq m, \sigma_0(i) < k \}|, & 0 > l > -m \\
2k-1 - |\{i: r_i \geq k, i \notin (m,l), \sigma_0(i) <k\}|, & l < -m. \end{cases} \end{equation} Then $$\sor_{\bf{r}}(\sigma, \sigma_0) = \sum_{k=1}^n b_k.$$

As before, $\sor_{\bf{r}}(\sigma, \sigma_0)$ in general depends not only on $\sigma_0$, but on $\bf{r}$ as well. However, the case when $\sigma_0$ is the identity permutation is an exception. 

 \begin{lemma} \label{permlemma} Let $\bf{r}$  be a non-decreasing sequence of integers $1 \leq r_1 \leq r_2 \leq \cdots \leq r_n \leq n$ with $r_k \geq k$, for all $k$. Let $\sigma \in B_{\bf{r}}$.  Then $$\sor_{\bf{r}}(\sigma, \text{\bf{id}}) = \sor(\sigma).$$
 \end{lemma}
 \begin{proof}  The case $|l| > m$ in~\eqref{bkdef} cannot occur. Namely, in the case when $\sigma_0 = \text{\bf{id}}$, we have $m=k$ and if $|l|>k$, $\sigma_k^{-1}(l) = \sigma_0^{-1}(l) = l.$ This contradicts $l = \sigma_k^{-1}(k)$. Therefore, the definition of $b_k$ simplifies to
\[ b_k = \begin{cases} | \{ i : l \leq i < k \}|, & 0< l \leq k \\ 2k-1- | \{ i : l \leq i < k \}|, & 0 > l > -k. \end{cases} \] 
This is precisely the ``distance'' that $k$ travels when being placed in its correct position with the sorting algorithm.
  \end{proof}

Signed permutations can be decomposed into two types of cycles. The cycles can be of the form $(a_1, \dots, a_k)$ (this cycle also takes $\bar{a_1}$ to $\bar{a_2}$, etc.) or of the form $(a_1, \dots, a_k, \bar{a_1}, \dots, \bar{a_k})$, for $k \geq 1$ and all $a_1, \dots, a_k$ different. The former cycles are called balanced and the letter ones unbalanced. Let 
$$\Cyc_0(\sigma) =\{|k| : k \text{ is the minimal number in absolute value in a balanced cycle of } \sigma\},$$ 
$$\Cyc_1(\sigma) =\{|k| : k \text{ is the minimal number in absolute value in a unbalanced cycle of } \sigma\},$$
and let $\cyc_0(\sigma)=|\Cyc_0(\sigma)|$ and $\cyc_1(\sigma)=|\Cyc_1(\sigma)|$. For example, the permutation $\sigma = \bar{3} \bar{9} \bar{5} \bar{7} 1 \bar{6} \bar{4} \bar{8} 2$ can be decomposed into 
$\sigma = (1 \bar{3} 5) ( 2 \bar{9} \bar{2} 9) (4 \bar{7}) (6 \bar{6}) (8)$, so $\Cyc_0(\sigma) = \{1, 4, 8\}$ and $\Cyc_1(\sigma) =\{ 2, 6\}$.

\begin{corollary} \label{coradd} Let $ \sigma_0 \in B_{\bf{r}}$ with $\sigma(k) > 0$ for $k >0$. Then
\begin{equation} \label{oddeven} \sum_{\sigma \in B_{\bf{r}}} q^{\sor_{\bf{r}}(\sigma,\sigma_0)}  \prod_{i \in \Cyc_0(\sigma \sigma_0^{-1})}t_i \prod_{i \in \Cyc_1(\sigma \sigma_0^{-1})}s_i  =   \prod_{k=1}^n \left(t_i + (q+ q^{2k-h_k})[h_k-1]_q+ s_iq^{2k-1} \right). \end{equation}

\end{corollary}

\begin{proof} As in the monochromatic case, there is a natural correspondence between the cycles of the permutation $\sigma \sigma_0^{-1}$ and the cycles in the bicolored graph $(g_{\bf{r}}(\sigma), g_{\bf{r}}(\sigma_0))$.
\end{proof}

\begin{corollary} Let $ \sigma_0 \in B_{\bf{r}}$ with $\sigma(k) > 0$ for $k >0$. Then
\begin{equation} \sum_{\sigma \in B_{\bf{r}}} q^{\sor_{\bf{r}}(\sigma,\sigma_0)}  t^{\ell'_B(\sigma \sigma_0^{-1})}  =   \prod_{k=1}^n \left(1 + q[h_k-1]_q t    +q^{2k-h_k}[h_k]_q t\right). \end{equation} In particular,
\begin{equation} \sum_{\sigma \in B_{\bf{r}}} q^{\sor_{\bf{r}}(\sigma)}  t^{\ell'_B(\sigma)}  =   \prod_{k=1}^n \left(1 + q[h_k-1]_q t    +q^{2k-h_k}[h_k]_q t\right). \end{equation} 

\end{corollary}
\begin{proof} It is not difficult to see that the  reflection length of a balanced cycle $(a_1, \dots, a_k)$ is $k-1$, while the reflection length of an unbalanced cycle is $k$. Therefore, $\ell'_B(\sigma) = n - \cyc_0(\sigma)$. The result follows from~\eqref{oddeven}.
\end{proof}

The minimal number of terms in $\{(1 \bar{1})\} \cup \{(i \; i+1): 1 \leq i \leq n\}$ needed to express $\sigma \in B_n$ is called the  length of $\sigma$. It is known to be equal to the type $B_n$ inversion number $\inv_B$ given in~\eqref{invb}. 

\begin{lemma} \label{hm} Let $\sigma \in B_{\bf{r}}$ and $M = g_{\bf{r}}(\sigma)$. Then $$\inv_B (\sigma) =  \mathrm{mix}(M).$$
\end{lemma}
\begin{proof}
Clearly, $N(\sigma)$ is equal to the number of blue edges in $M$. Moreover, two arcs $M(c_i) \cdot c_i$ and $M(c_j) \cdot c_j$ with $i < j$ can be in three different relative positions. 
\begin{itemize}
\item[(i)] They form a nesting. If the right arc is red, we have $\sigma(i) > \sigma(j)$ but not $-\sigma(i) > \sigma(j)$. If the right arc is blue, we have $-\sigma(i) > \sigma(j)$, but not $\sigma(i) > \sigma(j)$.
\item[(ii)] They form a crossing. If the right arc is blue, we have both $\sigma(i) > \sigma(j)$ and $-\sigma(i) > \sigma(j)$. If the right arc is red, neither $\sigma(i) > \sigma(j)$ nor $\sigma(i) > \sigma(j)$ is true.
\item[(iii)] They form an alignment. This case is the same as (ii).
\end{itemize}
\end{proof}

For a signed permutation $\sigma$ we define the set of positive right-to-left minimum letters to be
\[\mathrm{Prlminl}(\sigma) = \{k : 0 < \sigma(k) < |\sigma(l)| \text{ for all } l>k \}.\] It is not difficult to see that $o_k \cdot g_{\bf{r}}(\sigma)(o_k)$ is a large red edge if and only if $k \in \mathrm{Prlminl}(\sigma)$. Therefore we get the following corollary.
\begin{corollary} Let $ \sigma_0 \in B_{\bf{r}}$ with $\sigma(k) > 0$ for $k >0$. Then 
\[  \sum_{\sigma \in B_{\bf{r}}} q^{\sor_{\bf{r}}(\sigma,\sigma_0)}  \prod_{i \in \Cyc_0(\sigma \sigma_0^{-1})}t_i  =  \sum_{\sigma \in B_{\bf{r}}} q^{\inv_B(\sigma)}  \prod_{i \in \mathrm{Prlminl}(\sigma)}t_i = \prod_{k=1}^n \left(t_i + q[h_k-1]_q   +q^{2k-h_k}[h_k]_q \right).\]
\end{corollary}
\begin{proof} The positive right-to-left minimum letters in the signed permutation $\sigma$ correspond to the openers of the long red edges in $g_{\bf{r}}(\sigma)$.  So, the result follows from Corollary~\ref{coradd}, Lemma~\ref{hm}, and Corollary~\ref{corb}.
\end{proof}
Recall that $$\mathrm{nmin}_B(\sigma)  = |\{ i: \sigma(i) > |\sigma(j) \text{ for some } j>i\}| + N(\sigma).$$ It is readily seen that
$$\mathrm{nmin}_B(\sigma) = n - |\mathrm{Prlminl}(\sigma)|.$$ 
\begin{corollary} Let $ \sigma_0 \in B_{\bf{r}}$ with $\sigma(k) > 0$ for $k >0$. Then
\begin{equation} \sum_{\sigma \in B_{\bf{r}}} q^{\mathrm{inv}_B(\sigma)}  t^{\mathrm{nmin}_B(\sigma)}  = \sum_{\sigma \in B_{\bf{r}}} q^{\sor_{\bf{r}}(\sigma,\sigma_0)}  t^{\ell'_B(\sigma \sigma_0^{-1})}. \end{equation}
\end{corollary}
Since $\mathrm{\bf{id}} \in B_{\bf{r}}$ for every sequence ${\bf{r}}$, this generalizes the equidistribution result~\eqref{petb} for signed permutations~\cite{Petersen}.


\subsection{Type $D_n$ permutations}

The type $D_n$ permutations can be defined as signed permutations with an even number of minus signs.  The type $D_n$ inversion number is defined as 
\[ \inv_D(\sigma) = |\{ 1 \leq i < j \leq n: \sigma(i) > \sigma(j) \}| + |\{ 1 \leq i < j \leq n: -\sigma(i) > \sigma(j)|.\] It is well known~\cite{BB} that $\inv_D(\sigma)$ is equal to the length of $\sigma$, i.e.,  the minimal number of transpositions in $\{(\bar{1} 2), (1 2), ( 2 3), \dots, (n-1 \; n)\}$ needed to express $\sigma$ and that
\[\sum_{\sigma \in D_n}q^{\inv_D(\sigma)} = \prod_{i=1}^{n-1} (1+q^i)[i+1]_q = [n]_q \cdot \prod_{i=1}^{n-1} [2i]_q.\] For type $D_n$ permutations, we define 
\[\mathrm{Prlminl}'(\sigma) =\{\sigma(k): 1 < \sigma(k) < |\sigma(l)| \text{ for all } l>k  \} .\]

If $\sigma = (i_1 j_1)\cdots(i_k j_k)$ is the unique factorization with $1 < j_1< \cdots < j_k$,  the type $D_n$ sorting index is defined to be
\[\sor_D(\sigma) = \sum_{s=1}^k \left(j_s - i_s -2 \cdot \chi(i_s <0) \right).\] Petersen~\cite{Petersen} proved that
\begin{equation} \label{petdsorinv} \sum_{\sigma \in D_n}q^{\sor_D(\sigma)} = \sum_{\sigma \in D_n}q^{\inv_D(\sigma)}.\end{equation}

The type $D_n$ permutations correspond to bicolored matchings with an even number of blue edges via the map $g_{\bf{r}}$. Therefore, results for type $D_n$ permutations analogous to those given for signed permutations can be derived by considering appropriate restrictions of the bijections $\varphi_2$ and $\phi_2$. Namely, the color of the first rise of a bicolored weighted Dyck path that corresponds to a bicolored matching with an even number of blue edges via any of these two bijections  is determined by the colors of the other rises. The type $D_n$ inversions correspond to a modified $\mathrm{mix}$ statistic of bicolored matchings. Let
$$\mathrm{mix}'(M) = \nn(M) + 2\cc_{\ast b} (M) + 2 \al_{\ast b}(M)  \hspace{1cm} \text{ and } \hspace{1cm} \mathrm{Longr}'(M) = \mathrm{Longr}(M) \backslash \{1\}.$$

 \begin{corollary} Let $D$ be a Dyck path with height sequence $(h_1, \dots, h_n)$. Then
  \begin{equation} \label{almost} \sum_{M} q^{\mathrm{mix}'(M)}  \prod_{i \in \mathrm{Longr}'(M)} t_i =  \prod_{i=2}^n \left( t_i + q [h_i-1] _q + q^{2i-h_i-1}[h_i]_q \right),\end{equation}
  where the sum is  over all bicolored matchings of type $D$ with an even number of blue edges.
 \end{corollary}
 
 \begin{proof} We use the bijection $\varphi_2$ restricted to the bicolored weighted Dyck paths with even number of blue edges. Similarly as in the proof of Theorem~\ref{thmsigned}, we get that 
\[ \sum_{M } q_1^{\nn_{*r}(M)}q_2^{\nn_{*b}(M)}q_3^{\cc_{*r}(M)}q_4^{\cc_{*b}(M)}q_5^{\al_{*r}(M)}q_6^{\al_{*b}(M)} \prod_{i \in \mathrm{Longr}'(M)} t_i,  \] where the sum is  over all bicolored matchings of type $D$ with an even number of blue edges, is equal to \[ \prod_{i=2}^{n} \sum_{k=1}^{h_i}(q_1^{k-1} q_3^{h_i-k} q_5^{i-h_i} t_i^{\delta_{k,1}} + q_2^{h_i-k} q_4^{k-1} q_6^{i-h_i}).\]  Setting $q_1 = q_2 =q$, $q_4 = q_6 = q^2$, and $q_3 = q_5 =1$ yields~\eqref{almost}. 
 \end{proof}

The type $D_n$ sorting index of a permutation corresponds to the following modified sorting index of bicolored matchings. The matching $M$ is sorted to $M_0$ as before and
\[ \sor'(M) = \sum_{k=2}^{n} \sor'_k(M, M_0)\] where, for $2 \leq k \leq n$,
\begin{equation} \label{sord}
\sor'_k(M, M_0) = \begin{cases}
\sor_k(M, M_0), &\text{ if }  \col(o_k, M_k) = 0 \\
\sor_k(M, M_0) -1 , &\text{ if }  \col(o_k, M_k) =1 .
 \end{cases}
 \end{equation}
 Also, let  $\Cyc'_0(M, M_0) = \Cyc_0(M, M_0) \backslash \{1\} $ and $\Cyc'_1(M, M_0) = \Cyc_1(M, M_0) \backslash \{1\} $.
 
 \begin{lemma} Let $M_0$ be a matching with all red edges of type $D$ and let $(h_1, \dots, h_n)$ be the height sequence of $D$. Suppose $M=\phi_2((w_1, \dots, w_n), (\epsilon_1, \dots, \epsilon_n))$ is a bicolored matching of type $D$ with an even number of blue edges, where $\phi_2$ depends on $M_0$. Then
 $$\sor'(M, M_0) = \sum_{k=2}^n \left( w_k + \epsilon_k( 2k-1-h_k) -1\right).$$
 \end{lemma}
 
 \begin{proof} It follows from the property~\eqref{rtf} of the bijection $\phi_2$ and~\eqref{sord}.
 \end{proof}
 
 Therefore, via the bijection $\phi_2$ we get the following multivariate generating function.

 \begin{corollary} Let $D$ be a Dyck path with height sequence $(h_1, \dots, h_n)$ and $M_0$ a matching of type $D$ with only red edges.
 \[ \sum_{M} q^{\sor'(M, M_0)}  \prod_{i \in \Cyc'_0(M, M_0)} t_i \prod_{i \in \Cyc'_1(M, M_0)} s_i=  \prod_{i=2}^n \left( t_i + ( q + q^{2i-h_i-1})[h_i-1]_q + q^{2_i-2} s_i) \right),\]
 where the sum is  over all bicolored matchings of type $D$ with an even number of blue edges.
 \end{corollary}
 
 The results for the bicolored matchings with even number of blue edges yield results for restricted permutations of type $D_n$. Let $$D_n({\bf{r}}) =\{ \sigma \in D_n : |\sigma(k)| \leq r_k, 1\leq k \leq n \}.$$ For $\sigma \in D_n$, define $\Cyc'_0(\sigma) = \Cyc_0(\sigma) \backslash \{1\} $ and $\Cyc'_1(\sigma) = \Cyc_1(\sigma) \backslash \{1\} $. 
 \begin{corollary} Let ${\bf{r}}: 1 \leq r_1 \leq \cdots r_n \leq n$ be a nondecreasing integer sequence such that $r_k \geq k$. Let $(h_1, \dots, h_n)$ be the height sequence of the corresponding Dyck path $D({\bf{r}})$. 
 \begin{equation}
\sum_{\sigma \in D_n({\bf{r}})} q^{\inv_D(\sigma)} \prod_{i \in \mathrm{Prlminl}'(\sigma)} t_i = \prod_{i=2}^n \left( t_i + q [h_i-1]_q + q^{2i-h_i-1} [h_i]_q  \right)
\end{equation}
Moreover, suppose that $\sigma_0 \in D_n({\bf{r}})$ with $\sigma(k) > 0$ for all $k \geq 1$. Then
\begin{equation}
\sum_{\sigma \in D_n({\bf{r}})} q^{\sor_D(\sigma \sigma_0^{-1})} \prod_{i \in \Cyc'_0(\sigma \sigma_0^{-1})} t_i  \prod_{i \in \Cyc'_1(\sigma \sigma_0^{-1})} s_i  =   \prod_{i=2}^n \left( t_i + (q + q^{2i-h_i-1}) [h_i-1]_q +  q^{2i-2} s_i \right).  \end{equation}
\end{corollary}
\begin{proof} Note that $\inv_D(\sigma) = \inv_B(\sigma) - N(\sigma) = \mathrm{mix}(M) - \mathrm{b}(M) = \mathrm{mix}'(M)$, where $M = g_{\bf{r}}(\sigma)$. \end{proof}

\begin{corollary} Let ${\bf{r}}: 1 \leq r_1 \leq \cdots r_n \leq n$ be a nondecreasing integer sequence such that $r_k \geq k$. Let $(h_1, \dots, h_n)$ be the height sequence of the corresponding Dyck path $D({\bf{r}})$. Also let $\sigma_0 \in D_n({\bf{r}})$ with $\sigma(k) > 0$ for all $k \geq 1$. Then
\begin{equation}
\sum_{\sigma \in D_n({\bf{r}})} q^{\sor_D(\sigma \sigma_0^{-1})} \prod_{i \in \Cyc'_0(\sigma \sigma_0^{-1})} t_i = \sum_{\sigma \in D_n({\bf{r}})} q^{\inv_D(\sigma)} \prod_{i \in \mathrm{Prlminl}'(\sigma)} t_i. \end{equation}

\end{corollary}
In particular, when $r_1 =\cdots = r_n =n$ we have $ D_n({\bf{r}}) = D_n$, $h_k =k$, and 
\begin{equation}
\sum_{\sigma \in D_n} q^{\inv_D(\sigma)} \prod_{i \in \mathrm{Prlminl}'(\sigma)} t_i = \sum_{\sigma \in D_n} q^{\sor_D(\sigma)} \prod_{i \in \Cyc'_0(\sigma)} t_i   =   \prod_{i=2}^n \left( t_i + q [i-1]_q +  q^{i-1}[i]_q \right).  \end{equation}

\section*{Acknowledgement} The author would like to thank Michelle Wachs and Jim Haglund for useful conversations.

\end{document}